\documentclass[12pt]{amsproc}
\newtheorem{theorem}{\sc Theorem}[section]
\newtheorem{lemma}[theorem]{\sc Lemma}

\newtheorem{corollary}[theorem]{\sc Corollary}

\begin{document}

\title[Finite-by-nilpotent groups]{Finite-by-nilpotent groups and a variation of the BFC-theorem}

\author{Pavel Shumyatsky }
\address{ Pavel Shumyatsky: Department of Mathematics, University of Brasilia,
Brasilia-DF, 70910-900 Brazil}
\email{pavel@unb.br}
\thanks{The author thanks the referee for helpful comments on an earlier version of the paper. This research was supported by the Conselho Nacional de Desenvolvimento Cient\'{\i}fico e Tecnol\'ogico (CNPq),  and Funda\c c\~ao de Apoio \`a Pesquisa do Distrito Federal (FAPDF), Brazil.}
\keywords{}
\subjclass[2010]{20F12,20F24}
\begin{abstract} 
For a group $G$ and an element $a\in G$ let $|a|_k$ denote the cardinality of the set of commutators $[a,x_1,\dots,x_k]$, where $x_1,\dots,x_k$ range over $G$. The main result of the paper states that a group $G$ is finite-by-nilpotent if and only if there are positive integers $k$ and $n$ such that $|x|_k\leq n$ for every $x\in G$. More precisely, if $|x|_k\leq n$ for every $x\in G$ then $\gamma_{k+1}(G)$ has finite $(k,n)$-bounded order. 
Further, in any group $G$ the set $FC_k(G)=\{x\in G;\ |x|_k<\infty\}$ is a subgroup and $\gamma_{k+1}(FC_k(G))$ is locally normal.

\end{abstract}

\maketitle

\section{Introduction}

Let $G$ be a group. Given $k\geq1$ and elements $x_1,\dots,x_k\in G$, we write $[x_1,\dots,x_k]$ for the left-normed commutator $[\dots[[x_1,x_2],x_3]\dots,x_k]$. Elements which can be written as $[x_1,\dots,x_k]$ for suitable $x_1,\dots,x_k\in G$ will be called $\gamma_k$-values. Of course the subgroup generated by the $\gamma_k$-values is the $k$th term of the lower central series of $G$, usually denoted by $\gamma_k(G)$. For $a\in G$ write $$|a|_k=|\{[a,x_1,\dots,x_k];\ x_1,\dots,x_k\in G\}|.$$ Thus, $|a|_k$ is the cardinality of the set of $\gamma_{k+1}$-values starting with the element $a$. Note that $|a|_1$ is precisely the cardinality of the conjugacy class $a^G$. The group $G$ is a BFC-group if there is a number $n$ such that $|x^G|\leq n$ for every $x\in G$. B. H. Neumann proved in \cite{bhn} that if $G$ is a BFC-group in which $|x^G|\leq n$ for every $x\in G$, then the commutator subgroup $G'$ has finite $n$-bounded order. Here we will extend the Neumann theorem in the following way.

\begin{theorem}\label{main} Let $n$ be a positive integer and $G$ a group such that $|x|_k\leq n$ for every $x\in G$. Then $\gamma_{k+1}(G)$ has finite $(k,n)$-bounded order.
\end{theorem}
Throughout the article we use the expression ``$(a, b,\ldots)$-bounded'' to to mean that a quantity is finite and bounded by a certain number depending only on the parameters $a,b,\ldots$.

The next corollary is straightforward.
\begin{corollary}\label{glav} A group $G$ is finite-by-nilpotent if and only if there are positive integers $k$ and $n$ such that $|x|_k\leq n$ for every $x\in G$. 
\end{corollary}

The proof of Theorem \ref{main} is based on some new tools that were recently developed within the framework of the probabilistic group theory (see the Section 3 for more details).

Recall that the FC-centre of a group $G$ is the subgroup $FC(G)$ formed by the elements $x\in G$ such that $|x^G|<\infty$. This is a characteristic subgroup of $G$ whose commutator subgroup is locally normal, that is, any finite subset of $FC(G)'$ is contained in a finite normal subgroup \cite[14.5]{rob}. For $k\geq1$ define $$FC_k(G)=\{x\in G;\ |x|_k<\infty\}.$$ It turns out that $FC_k(G)$ is a subgroup of $G$. 
\begin{theorem}\label{main3} In any group $G$ the set $FC_k(G)$ is a subgroup and $\gamma_{k+1}(FC_k(G))$ is locally normal.
\end{theorem}
We will deduce the above theorem from the following result which may be of independent interest.
\begin{theorem}\label{main2} Assume that $G$ is a group generated by finitely many elements $a_1,\dots,a_r$ such that $|a_i|_k\leq n$ for any $i=1\dots,r$. Then $\gamma_{k+1}(G)$ has finite $(k,n,r)$-bounded order.
\end{theorem}

Of course, in general an FC-group need not be a BFC-group. For example a direct product of infinitely many non-abelian finite groups is FC but not BFC. On the other hand, Shalev observed in \cite{shalev} that a profinite FC-group is necessarily BFC. Here we will prove the following related result.

\begin{theorem}\label{profinite} Assume that $G$ is a profinite group such that $|x|_k<\infty$ for any $x\in G$. Then $\gamma_{k+1}(G)$ is finite.
\end{theorem}

{\bf Notation.} The notation used in the paper is mostly standard. For a group $G$ and an integer $i\geq1$ we let $X_i$ denote the set of $\gamma_i$-values in $G$. If $a\in G$, we denote by $X_i(a)$ the set of all commutators of the form $[a,x_1,\dots,x_i]$. Thus $X_i(a)$ is the set $\gamma_{i+1}$-values which can be written as a commutator starting with the element $a$. Obviously, $|a|_k=|X_k(a)|$.

For elements $a,b\in G$ we write $[b,{}_la]$ to denote the long commutator $[b,a,\dots,a]$, where $a$ is repeated $l$ times. If $N$ is a subgroup of $G$, the symbol $[N,a]$ stands for the subgroup generated by the commutators $[x,a]$, where $x$ ranges over $N$. For $l\geq2$ define by induction $[N,{}_la]=[[N,{}_{l-1}a],a]$. Note that if $N$ is an abelian normal subgroup, $[N,{}_la]$ is precisely the set of commutators $[x,{}_la]$ where $x\in N$.

\section{Preliminaries}

In this section we collect some helpful facts that will be required in the proofs of the main results. We start with the following well-known criterion for nilpotency of a group, due to P. Hall \cite{hall1}.
\begin{theorem}\label{hall}
Let $N$ be a nilpotent normal subgroup of a group $G$ such that $N$ and $G/N'$ are both nilpotent. Then $G$ is nilpotent of class bounded in terms of the classes of $N$ and $G/N'$.
\end{theorem}

The following lemma will be useful.

\begin{lemma}\label{redufini}
Let $G$ be a group and $k,s$ positive integers.
\begin{enumerate}
\item $|\gamma_k(G)|\leq s$ if and only if $|\gamma_k(H)|\leq s$ for any finitely generated subgroup $H\leq G$.
\item If $G$ is residually finite, $|\gamma_k(G)|\leq s$ if and only if $|\gamma_k(Q)|\leq s$ for any finite quotient $Q$ of $G$.
\end{enumerate}
\end{lemma}
\begin{proof} This is immediate from \cite[Lemma 4.4]{ebeshu}.
\end{proof}

\begin{lemma}\label{coprime}
Let $A$ be a finite group acting on a finite group $G$ such that $(|A|,|G|)=1$. Then $[G,A,A]=[G,A]$.
\end{lemma}

\begin{lemma}\label{lll}
Let $G$ be a metabelian group and suppose that $a,b\in G$ are $l$-Engel elements. If $x\in\langle a,b\rangle$, then $x$ is $(2l+1)$-Engel.
\end{lemma}
\begin{proof} Note that both subgroups $G'\langle a\rangle$ and $G'\langle b\rangle$ are normal in $G$ and nilpotent of class at most $l$. Thus, by Fitting theorem, the subgroup $G'\langle a,b\rangle$ is nilpotent of class at most $2l$ and the lemma follows.
\end{proof}

The following observation will be useful.
\begin{lemma}\label{yes} Suppose that $a$ is an element of a group $G$ such that $|a|_k\leq n$. Then $[G:C_G(b)]\leq n$ for any $b\in X_{k-1}(a)$. Moreover, $[G:C_G(d)]\leq n^2$ for any $d\in X_k(a)$. Consequently, the index $[G:C_G(X_k(a))]$ is at most $n^3$.
\end{lemma}
\begin{proof} The centralizer of every element in $X_{k-1}(a)$ has index at most $n$ in $G$ because the set $[X_{k-1}(a),G]$ has at most $n$ elements. Choose $d\in X_k(a)$. Since $d$ is a product of an element from $X_{k-1}(a)$ and an inverse of a conjugate of such element, it follows that $[G:C_G(d)]\leq n^2$. By hypotheses, $|X_k(a)|\leq n$ and so there are at most $n^3$ conjugates of elements from $X_k(a)$ in $G$. 
\end{proof}

\section{Groups in which conjugacy classes of $\gamma_k$-values are of bounded size}

The purpose of this section is to establish the following result.

\begin{theorem}\label{dadada} Let $k\geq1$ and $G$ be a group in which $|x^G|\leq n$ for any $x\in X_k$. Then $G$ has a nilpotent subgroup of $(k,n)$-bounded index and $(k,n)$-bounded class. 
\end{theorem}

It is unclear whether under the hypotheses of the above theorem $G$ in fact has a nilpotent subgroup of $(k,n)$-bounded index and $k$-bounded class. If $k=1$, by B. H. Neumann's theorem $G$ has a class-two nilpotent normal subgroup of $n$-bounded index. If $k=2$, the main result of \cite{ebeshu} implies that $G$ has a class-four nilpotent normal subgroup of $n$-bounded index.

If $K$ is a subgroup of a finite group $G$, the commuting probability $Pr(K,G)$ of $K$ in $G$ is the probability that a random element of $K$ commutes with a random element of $G$. Note that if $|x^G|\leq n$ for every $x\in K$, then $Pr(K,G)\geq\frac{1}{n}$. It was shown in \cite{DS} (see also the asymmetric Neumann theorem in \cite[Theorem 4.3]{ebeshu}) that if $Pr(K,G)\geq\epsilon>0$, then there is a normal subgroup $T\leq G$ and a subgroup $B\leq K$ such that the indices $[G:T]$ and $[K:B]$ and the order of the commutator subgroup $[T,B]$ are $\epsilon$-bounded.

Recall that an element $x$ of a group $G$ is called (left) $l$-Engel if $[y,{}_lx]=1$ for every $y\in G$. The group $G$ is said to be $l$-Engel if so are the elements of $G$.

\begin{proof}[Proof of Theorem \ref{dadada}] As was mentioned, the result holds true if $k\leq2$ so we assume that $k\geq3$. By \cite{dms} the derived group $\gamma_k(G)'$ has $n$-bounded order. Factoring out $\gamma_k(G)'$ we may assume that $\gamma_k(G)$ is abelian and so $G$ is abelian-by-nilpotent. The classical theorem of Hall \cite{hall2} states that finitely generated abelian-by-nilpotent groups are residually finite. Thus, by virtue of Lemma \ref{redufini}, it is sufficient to prove the result for finite quotients of finitely generated subgroups of $G$. Therefore without loss of generality we may assume that $G$ is finite and $\gamma_k(G)$ is abelian.

Assume first that $G$ is metabelian. Let $N$ be a maximal abelian normal subgroup containing $G'$. Set $${\bf G}=G\times\dots\times G\ \ (k-1\text{ factors}).$$ If ${\bf g}=(g_1,\dots,g_{k-1})\in{\bf G}$, the subgroup $N_{\bf g}=[N,g_1,\dots,g_{k-1}]$ consists of $\gamma_k$-values, that is $N_{\bf g}\subseteq X_k$. Hence, $Pr(N_{\bf g},G)\geq\frac{1}{n}$. By \cite{DS} there is a normal subgroup $T_{\bf g}$ and a subgroup $B_{\bf g}\leq N_{\bf g}$ such that the indices $[G:T_{\bf g}]$ and $[N_{\bf g}:B_{\bf g}]$ and the order of $[T_{\bf g},B_{\bf g}]$ are $n$-bounded. Note that $[T_{\bf g},B_{\bf g}]$ is normal in $T_{\bf g}$ and so the normal closure $\langle[T_{\bf g},B_{\bf g}]^G\rangle$ has $n$-bounded order. We emphasize that the bounds here do not depend on the choice of ${\bf g}\in{\bf G}$.

 The proof now breaks into several cases.

1. The case where $G$ is a metabelian $p$-group. It follows that there is an $n$-bounded number $e$ such that $Z_e(G)$ contains $\langle[T_{\bf g},B_{\bf g}]^G\rangle$ for every ${\bf g}\in {\bf G}$. Here and throughout $Z_i(H)$ denotes the $i$th term of the upper central series of a group $H$. Pass to the quotient $G/Z_e(G)$ and assume that $[T_{\bf g},B_{\bf g}]=1$ for any ${\bf g}\in {\bf G}$. As the index $[N_{\bf g}:B_{\bf g}]$ is $n$-bounded, there are $n$-boundedly many $\gamma_k$-values $x_1,\dots,x_j$ such that $N_{\bf g}=\langle B_{\bf g},x_1,\dots,x_j\rangle$. Set $C_{\bf g}=C_G(N_{\bf g})$. Since $C_{\bf g}$ contains the intersection of $T_{\bf g}$ and all centralizers $C_G(x_i)$, it follows that the index of $C_{\bf g}$ is $n$-bounded. Since $G$ is metabelian, the subgroups $N_{\bf g}$ and $C_{\bf g}$ are normal in $G$. 

For $g\in G$ and $i=1,2,\dots$ write $N_{i,g}$ for $[N,{}_ig]$ and $C_{i,g}$ for $C_G(N_{ig})$. 
Let $\beta_i$ be the maximal index of $C_{i,g}$ , where $g$ ranges over $G$. It is clear that $\beta_1\geq\beta_2\geq\beta_3\geq\dots$. Since $\beta_{k-1}$ is $(k,n)$-bounded, there is a $(k,n)$-bounded number $u$ such that $k-1\leq u$ and $\beta_u=\beta_{2u}$.

Choose $g\in G$ such that $C_{u,g}=C_{2u,g}$ is of index $\beta_u$ in $G$. We have $C_{u,g}=C_{j,g}$ for $j=u,u+1,\dots,2u$. Let $h\in C_{u,g}$. Note that $$N_{j,g}=[N_{u,g},{}_{j-u}g]=[N_{u,g},{}_{j-u}gh].$$ In particular, $$N_{(2u),g}=[N_{u,g},{}_u(gh)].$$ This shows that $C_{u,g}=C_{(2u),g}$ contains $C_{u,gh}$. Since $\beta_u$ is maximal and since $h$ was chosen in $C_{u,g}$ arbitrarily, we are forced to conclude that $$C_{u,g}=C_{u,gh} \text{ whenever } h\in C_{u,g}.$$ Set $C_{u,g}=D$ and note that $D$ centralizes $N_{u,gh}$ for any $h\in D$. 

Obviously, $D$ is a normal subgroup of index $\beta_u$ in $G$. Let $\overline{G}=G/Z(D)$. Since $Z(D)$ contains $N_{u,(gh)}=[N,(gh),\dots,(gh)]$, we conclude that $\overline{gh}$ is $(u+k)$-Engel in $\overline{G}$ for any $h\in D$. We deduce from Lemma \ref{lll} that $D$ is $l$-Engel for $l=2(u+k)+2$. Let us show that $D$ is nilpotent of $(k,n)$-bounded class.

Fix ${\bf g}$ and let $\bar D$ denote the image of $D$ in $G/C_{\bf g}$. Note that $\bar D$ naturally acts on $N_{\bf g}$ and the semidirect product of $N_{\bf g}$ by $\bar D$ is an $l$-Engel group having an abelian normal subgroup of $n$-bounded index. So it is nilpotent of $(k,n)$-bounded class, say $c$. This means that $N_{\bf g}\leq Z_c(D)$ for each ${\bf g}\in{\bf G}$. In particular, $[G',g_1,\dots,g_{k-1}]\leq Z_c(D)$ and therefore $\gamma_{k+1}(G)\leq Z_c(D)$. It follows that $D$ is nilpotent of class at most $k+c$.

Thus, in the case where $G$ is a metabelian $p$-group the result follows. 

2. The coprime case. Here we assume that $G/N$ is an abelian $p$-group while $N$ is a $p'$-group. Recall that $\langle[T_{\bf g},B_{\bf g}]^G\rangle$ has $n$-bounded order for any ${\bf g}\in {\bf G}$. It follows that $C_G(\langle[T_{\bf g},B_{\bf g}]^G\rangle)$ has $n$-bounded index in $G$ and so we may assume that $T_{\bf g}$ centralizes $\langle[T_{\bf g},B_{\bf g}]^G\rangle$. Then, because of Lemma \ref{coprime}, $T_{\bf g}$ centralizes $B_{\bf g}$. Similarly we may assume that $T_{\bf g}$ centralizes $N_{\bf g}/B_{\bf g}$, whence, again because of Lemma \ref{coprime}, $T_{\bf g}$ centralizes $N_{\bf g}$. So we simply assume that $B_{\bf g}=N_{\bf g}$ and $T_{\bf g}=C_G(N_{\bf g})$. Note that if $g\in G$ and ${\bf g}=(g,\dots,g)\in {\bf G}$, then $N_{\bf g}=[N,g]$ so in what follows we write $N_g$ in place of $N_{\bf g}$ and $T_g$ in place of $T_{\bf g}$.

Choose $g\in G$ such that the index $[G:T_g]$ is maximal. Let $h\in T_g$. Because of Lemma \ref{coprime} we have $$N_g=[N_g,g]=[N_g,gh]\leq N_{gh}.$$ It follows that $T_{gh}\leq T_g$ for any $h\in T_g$. Since $[G:T_g]$ is maximal, deduce that $T_{gh}=T_g$ for any $h\in T_{\bf g}$. So $T_g$ centralizes both $N_g$ and $N_{gh}$. It is easy to see that $N_h\leq N_gN_{gh}$, so $T_g$ centralizes $N_h$. This holds for every $h\in T_g$ and so $[N,T_g,T_g]=1$. In view of Lemma \ref{coprime} deduce that $[N,T_g]=1$. Therefore $T_g$ is abelian and this proves the theorem in the coprime metabelian case. 

3. The general metabelian case. Let $e=n!$. Then $G^e$ centralizes $\gamma_k(G)$ and so is nilpotent of class at most $k$. For each prime $p\leq n$ let $G_p$ be the preimage in $G$ of the Sylow $p$-subgroup of $G/N$. For $p>n$, the Sylow $p$-subgroups of $G$ are contained in $G^e$. Hence $G=G^e\prod_{p\leq n}G_p$. Therefore by Fitting's theorem it suffices to prove the result for $G=G_p$. Thus we may assume that $G/N$ is a $p$-group. Let P be a Sylow p-subgroup of G. Then $G=NP$. By the result on $p$-groups, we can replace $P$ with a subgroup of bounded index and bounded class, so without loss of generality $P$ has $(k,n)$-bounded class $c$ say. It follows that $[P\cap N,{}_cG]=1$ and  so $P\cap N\leq Z_c(G)$. Factoring out $Z_c(G)$, we may thus assume that $N$ is a $p'$-group. By the result in the coprime case we are done.

The general case. We now drop the assumption that $G$ is metabelian. Since the derived length of $G$ is at most $k+1$, we can use the induction on the derived length. By induction, $G'$ has a nilpotent subgroup $A$ of $(k,n)$-bounded index and  $(k,n)$-bounded class. In view of \cite{khuma} we can assume that $A$ is characteristic in $G'$ and therefore normal in $G$. Applying Theorem \ref{hall} we can assume that $A$ is abelian.

If $A=G'$, then $G$ is metabelian and the result holds. We therefore assume that $A<G'$. Further, replacing $G$ by $C_G(G'/A)$ we have that $G/A$ is nilpotent of class at most 2. Fix a set $X$ of size at most $|G'/A|$ such that $G'$ is generated by $A$ and $X$. Set $${\bf X}=X\times\dots\times X\ \ (k-1\text{ factors}).$$ If ${\bf x}=(x_1,\dots,x_{k-1})\in{\bf X}$, the subgroup $A_{\bf x}=[A,x_1,\dots,x_{k-1}]$ consists of $\gamma_k$-values, that is $A_{\bf x}\subseteq X_k$. Hence, $Pr(A_{\bf x},G)\geq\frac{1}{n}$. By \cite{DS} there is a normal subgroup $T_{\bf x}$ and a subgroup $B_{\bf x}\leq A_{\bf g}$ such that the indices $[G:T_{\bf x}]$ and $[A_{\bf x}:B_{\bf x}]$ and the order of $[T_{\bf x},B_{\bf x}]$ are $n$-bounded. Note that the subgroup $A_{\bf x}$ is normal in $G$. We therefore can choose $B_{\bf x}$ normal in $G$. Let $H_{\bf x}$ denote the stabilizer of the series $$1\leq[T_{\bf x},B_{\bf x}]\leq B_{\bf x}\leq A_{\bf x}.$$ The subgroup $H_{\bf x}$ contains the intersection of $$C_G([T_{\bf x},B_{\bf x}])\cap T_{\bf x}\cap C_G(A_{\bf x}/B_{\bf x}).$$ We therefore deduce that $H_{\bf x}$ has $(k,n)$-bounded index in $G$. Let $$H=\cap_{{\bf x}\in{\bf X}}H_{\bf x}.$$

Since $|{\bf X}|\leq|G'/A|^{k-1}$, observe that the index of $H$ in $G$ is $(k,n)$-bounded. Therefore it is sufficient to show that $H$ has a nilpotent subgroup of $(k,n)$-bounded index and $(k,n)$-bounded class. Note that $A_{\bf x}\leq Z_3(H)$ for any ${\bf x}\in{\bf X}$. It follows that $H'$ is nilpotent of class at most $k+3$. An application of Hall's Theorem \ref{hall} reduces the problem to the case where $H'$ is abelian. We already know that for metabelian groups the theorem holds. This completes the proof.
\end{proof}

\section{Proof of Theorem \ref{main}}

Let $G$ be a group satisfying the hypotheses of Theorem \ref{main}. We need to show that the order of $\gamma_{k+1}(G)$ is bounded in terms of $k$ and $n$ only. It is immediate from Lemma \ref{yes} that $|x^G|\leq n$ for any $\gamma_k$-value $x\in X_k$. By \cite{dms} the derived group $\gamma_k(G)'$ has $n$-bounded order. Factoring out $\gamma_k(G)'$ we may assume that $\gamma_k(G)$ is abelian and so $G$ is abelian-by-nilpotent. The combination of  Hall's theorem \cite{hall2} with Lemma \ref{redufini} enables us to assume that $G$ is finite.

Now we wish to show that the nilpotent residual $M=\gamma_\infty(G)$ has $(k,n)$-bounded order. Evidently, it is sufficient to show that the orders of the Sylow $p$-subgroups $P$ of $M$ are uniformly $(k,n)$-bounded. Note that since $G$ is metanilpotent, $P=[P,H]$ where $H$ is a Hall $p'$-subgroup of $G$ (see for example \cite[Lemma 2.4]{ast}). Denote by $F=F(G)$ the Fitting subgroup of $G$. According to Theorem \ref{dadada} the index $[G:F]$ is $(k,n)$-bounded so we will use induction on $[G:F]$. If $[G:F]=1$, then $G$ is nilpotent and $P=1$. We therefore assume that $F\neq G$. Since $G$ is soluble, there is a normal subgroup $T<G$ such that $F\leq T$ and the index $[G:T]=q$ is a prime. As the index of $F$ in $T$ is strictly smaller than in $G$, by induction the order of $\gamma_\infty(T)$ is $(k,n)$-bounded. Passing to the quotient $G/\gamma_\infty(T)$ we can assume that $\gamma_\infty(T)=1$ and so $F=T$ has prime index $q$ in $G$. Remark that under our assumptions necessarily $p\neq q$. Choose a $q$-element $u\in G$ such that $G=F\langle u\rangle$. Note that $$P=[P,u]=[P,{}_ku].$$ Since $P$ is abelian, for every $g\in P$ we have $[g,u]=[u,g^{-1}]$. Deduce that $[P,{}_ku]=[u,P,{}_{k-1}u]\leq X_k(u)$. This shows that the order of $P$ is at most $n$. So indeed $M$ has $(k,n)$-bounded order. Passing to the quotient $G/M$ without loss of generality assume that $G$ is nilpotent.

Our next goal is to show that the nilpotency class of $G$ is $(k,n)$-bounded.  By Theorem \ref{dadada} $G$ has a nilpotent normal subgroup $L$ of $(k,n)$-bounded index and $(k,n)$-bounded class. If $G=L$ we have nothing to prove and so we assume that $L$ is a proper subgroup of $G$. Arguing by induction on the index $[G:L]$ we assume that any proper normal subgroup $B$ such that $L\leq B<G$ is nilpotent of $(k,n)$-bounded class. Therefore without loss of generality we can assume that $L$ has prime index.

By virtue of Hall's Theorem \ref{hall} we can assume that $L$ is abelian. Let $g\in G\setminus L$ so that $G=L\langle g\rangle$. Choose $l\in L$. Taking into account that $L$ is an abelian normal subgroup write $$[g,{}_kgl]=[g,l,{}_{k-1}g]=[l^{-1},{}_kg].$$ The set $K$ of all commutators $[l^{-1},{}_kg]$, where $l$ ranges over $L$, is a subgroup. By hypothesis there are at most $n$ elements of the form $[g,{}_kgl]$ so we conclude that the order of the subgroup $K$ is at most $n$. Factoring out $K$ we may assume that $[L,{}_kg]=1$, in which case $G$ is nilpotent of class at most $k$. Hence, we have proved that $G$ has $(k,n)$-bounded class.

It remains to show that $\gamma_{k+1}(G)$ has $(k,m)$-bounded order. We will use induction on the nilpotency class of $G$. Thus, by induction, $\gamma_{k+1}(G/Z(G))$ has $(k,m)$-bounded order. By \cite[Theorem B]{guma} and the remark in the end of \cite{guma} deduce that $\gamma_{k+2}(G)$ has $(k,m)$-bounded order. We factor out $\gamma_{k+2}(G)$ and without loss of generality assume that $\gamma_{k+1}(G)$ is in the centre of $G$. For $a\in G$ and $i=1,\dots,k$ let $Y_i(a)$ denote the set of elements of the form $[a,x_1,\dots,x_{k+1-i}]$, where at least one of the elements $x_1,\dots,x_{k+1-i}$ belongs to $X_i$. Note that $Y_1(a)$ is precisely the set $X_k(a)$. By hypothesis, $|Y_1(a)|\leq n$ for any $a\in G$.

We claim that for any $i\leq k$ and $a\in G$ the size of the set $Y_i(a)$ is at most $n^{2^{i-1}}$. 

This will be established by induction on $i$. Since the claim holds for $i=1$, assume that $i\geq2$. Choose $[a,x_1,\dots,x_{k+1-i}]\in Y_i(a)$. At least one of the elements $x_1,\dots,x_{k+1-i}$ lies in $X_i$. For the sake of clarity without loss of generality assume that $x_{k+1-i}\in X_i$ and write $x_{k+1-i}=[y_1,y_2]$ for suitable $y_1\in X_{i-1}$ and $y_2\in G$. Set $x=[a,x_1,\dots,x_{k-i}]$. We have $$[a,x_1,\dots,x_{k+1-i}]=[x,x_{k+1-i}]=[x,[y_1,y_2]].$$ Taking into account that $\gamma_{k+1}(G)\leq Z(G)$ we obtain $$[x,[y_1,y_2]]=[x,y_1,y_2][y_2,x,y_1]=[x,y_1,y_2][x,{y_2}^{-1},y_1].$$ Observe that both factors $[x,y_1,y_2]$ and $[x,{y_2}^{-1},y_1]$ belong to $Y_{i-1}(a)$, which by induction is of size at most $n^{2^{i-2}}$. Therefore $[a,x_1,\dots,x_{k+1-i}]$ can take at most $n^{2^{i-1}}$ values, as claimed.

We can now complete the proof of Theorem \ref{main}. The fact that $|Y_k(a)|\leq n^{2^{k-1}}$ means that $|a^{X_k}|\leq n^{2^{k-1}}$ for any $a\in G$. According to \cite[Corollary 1.3]{glasgow} this implies that $\gamma_{k+1}(G)$ has finite $(k,n)$-bounded order, as required. This completes the proof of Theorem \ref{main}.

\section{$FC_k$-groups} 

Here we will prove Theorems \ref{main3} and \ref{main2}. We will start with some general observations concerning elements in $FC_k(G)$. The next lemma shows that $FC_k(G)$ is a subgroup.

\begin{lemma}\label{product} Let $G$ be a group and assume that $a,b\in FC_k(G)$. Then $|a^{-1}|_k=|a|_k$ and  $|ab|_k\leq{|a|_k}^3|b|_k$. 
\end{lemma}
\begin{proof} For any $x\in G$ we have $[a^{-1},x]=([a,x]^{-1})^{a^{-1}}$.  A straightforward induction on $k$ shows that for any $x_1,\dots,x_k$ there are elements $t_1,\dots,t_k\in G$ such that $$[a^{-1},x_1,\dots,x_k]=[a,t_1,\dots,t_k]^{-1},$$ whence we deduce that $X_k(a^{-1})=X_k(a)^{-1}$ and so $|a^{-1}|_k=|a|_k$.

To prove the other statement, choose $g_1,\dots,g_k\in G$. We have $[ab,g_1]=[a,g_1]^b[b,g_1]$. By induction on $k$ one easily proves that there are elements $a_1\in a^G$ and $u_1,\dots,u_k,v_1,\dots,v_k$ such that $$[ab,g_1,\dots,g_k]=[a_1,u_1,\dots,u_k][b,v_1,\dots,v_k].$$ Suppose that $|a|_k\leq n$. Lemma \ref{yes} shows that in the above equality the factor $[a_1,u_1,\dots,u_k]$ can take at most $n^3$ values while by hypotheses the factor $[b,v_1,\dots,v_k]$ can take at most $|b|_k$ values, whence the result.
\end{proof}

\begin{lemma}\label{generation} Let $G$ be a group generated by a set $S$. Then $\gamma_{k+1}(G)$ is generated by the conjugates of $X_k(g)$, where $g$ ranges over $S$.
\end{lemma}
\begin{proof} Let $N=\langle X_k(g)^G \vert g\in S\rangle$. Obviously $N\leq\gamma_{k+1}(G)$. Moreover, the image of any $g\in S$ in $G/N$ belongs to $Z_k(G/N)$. It follows that $G/N$ is nilpotent of class at most $k$ and so $\gamma_{k+1}(G)\leq N$.
\end{proof}

\begin{lemma}\label{abeliansubgr} Assume the hypotheses of Theorem \ref{main2} and let $N$ be an abelian normal subgroup of $G$. Then $[N,{}_kG]$ has finite $(k,n,r)$-bounded order. 
\end{lemma}
\begin{proof} If $k=1$, then $G$ is generated by $r$ elements whose centralizers have index at most $n$. It follows that the centre of $G$ has index at most $n^r$ and so, by Schur's theorem \cite[Theorem 10.1.4]{rob}, $G'$ has finite $(n,r)$-bounded order. 

We therefore assume that $k\geq2$. Let $A=\{a_1,\dots,a_r\}$ and ${\bf A}=A\times A\times\dots\times A$, where $A$ occurs $k$ times. For ${\bf b}=(b_1,\dots,b_k)\in{\bf A}$ let $N_{\bf b}$ stand for $[N,b_1,\dots,b_k]$. Note that since $N$ is abelian, $N_{\bf b}$ is a subgroup. If $g\in N$, note that $[g,b_1]=[b_1,g^{-1}]$. Therefore $$N_{\bf b}=[b_1,N,b_2\dots,b_k]\leq X_k(b_1).$$ We conclude that $N_{\bf b}$ is finite of order at most $n$. Let $N_0$ be the product of all conjugates of $N_{\bf b}$, where ${\bf b}$ ranges over ${\bf A}$. The above observations combined with Lemma \ref{yes} imply that $N_0$ is finite and has $(k,n,r)$-bounded order. Obviously, $N/N_0\leq Z_k(G/N_0)$. This completes the proof.
\end{proof}

\begin{proof}[Proof of Theorem \ref{main2}] Let $N=\gamma_{k+1}(G)$. Lemma \ref{generation} shows that $N$ is generated by conjugates of the set $\cup_iX_k(a_i)$. It follows from Lemma \ref{yes} that $C_G(N)$ has finite $(k,n,r)$-bounded index in $G$. In particular, the index $[N:Z(N)]$ is $(k,n,r)$-bounded and so by Schur's theorem $N'$ has finite $(k,n,r)$-bounded order. We pass to the quotient $G/N'$ and assume that $N$ is abelian. Now Lemma \ref{abeliansubgr} tells us that $[N,{}_kG]$ has finite $(k,n,r)$-bounded order. Pass to the quotient $G/[N,{}_kG]$ and assume that $[N,{}_kG]$=1. Thus, $G$ is nilpotent of class at most $2k$. Arguing by induction on the class of $G$ assume that the image of $N$ in $G/Z(G)$ has finite $(k,n,r)$-bounded order. By \cite{guma} $\gamma_{k+2}(G)$ has finite $(k,n,r)$-bounded order so we can factor out $\gamma_{k+2}(G)$ and assume that $N\leq Z(G)$. Since $N$ is generated by the set $\cup_iX_k(a_i)$, which is of bounded size, it is sufficient to show that any element of $X_k(a_i)$ has order at most $n$. Thus, choose $a\in\{a_1,\dots,a_r\}$ and $g_1,\dots,g_k\in G$. Keeping in mind that $G$ is nilpotent of class $k+1$ for any $i\geq1$ write $$[a,g_1,\dots,g_k]^i=[a,g_1,\dots,g_k^i]\in X_k(a).$$ This shows that $X_k(a)$ is closed under taking powers of its elements. We conclude that any element in $X_k(a)$ has order at most $n$. The theorem follows.
\end{proof}

We now deal with Theorem \ref{main3}.

\begin{proof}[Proof of Theorem \ref{main3}] Let $G$ be a group and set $T=FC_k(G)$. It is immediate from Lemma \ref{product} that $T$ is a subgroup. To show that $\gamma_{k+1}(T)$ is locally normal, choose a finite subset $E\subseteq\gamma_{k+1}(T)$. There are finitely many elements $a_1,\dots,a_r\in T$ such that $E\subseteq\gamma_{k+1}(\langle a_1,\dots,a_r\rangle)$. Set $A=\langle a_1,\dots,a_r\rangle$.  Theorem \ref{main2} says that $\gamma_{k+1}(A)$ is finite. Moreover the combination of Lemma \ref{generation} and Lemma \ref{yes} shows that $\gamma_{k+1}(A)$ is generated by elements whose centralizers have finite index in $G$. We deduce that $\gamma_{k+1}(A)$ is contained in a finite normal subgroup of $G$. The result follows.
\end{proof}

Finally we furnish a proof of Theorem \ref{profinite}.

\begin{proof}[Proof of Theorem \ref{profinite}] By hypotheses, $G$ is a profinite $FC_k$-group. For every positive integer $i$ set $$S_i=\{x\in G;\ |x|_k\leq i\}.$$ The sets $S_i$ are closed and their union is the whole group $G$. Therefore, by Baire category theorem \cite[Theorem 34]{kel} at least one of these sets has non-empty interior. It follows that there is a positive integer $j$, an element $a\in G$, and an open normal subgroup $N\leq G$ such that the coset $Na$ is contained in $S_j$. Since every element of $N$ can be written as a product $xy^{-1}$, where $x,y\in Na$, Lemma \ref{product} shows that $|g|_k\leq j^4$ for all $g\in N$. Choose a transversal $a_1,\dots,a_r$ of $N$ in $G$, and let $m$ be the maximum of $|a_1|_k,\dots,|a_r|_k$. Note that every element of $G$ can be written as a product $bg$, where $b\in\{a_1,\dots,a_r\}$ and $g\in N$. Now Lemma \ref{product} shows that $|x|_k\leq j^4m^3$ for any $x\in G$. Hence, by Theorem \ref{main}, $\gamma_{k+1}(G)$ has finite order.
\end{proof}

\end{document}